\documentclass[12pt]{amsart}
\usepackage{amssymb}
\usepackage{amsthm}
\usepackage{amsmath}
\usepackage{mdwlist}
\usepackage{mathrsfs}
\usepackage{graphicx}
\usepackage{a4wide}
\usepackage{thmtools}
\usepackage{mathtools,dsfont}
\usepackage{multicol}

\newtheorem{theorem}{Theorem}

\newtheorem{proposition}[theorem]{Proposition}
\newtheorem{corollary}[theorem]{Corollary}

\theoremstyle{definition}

\newtheorem{example}[theorem]{Example}
\declaretheorem[numbered=no]{Theorem}
\declaretheorem[numbered=no]{Theorem A}
\declaretheorem[numbered=no]{Theorem B}
\theoremstyle{remark}
\newtheorem{remark}[theorem]{Remark}

\newcommand{\R}{\mathbb{R}}

\newcommand{\N}{\mathbb{N}}

\newcommand{\p}{\varphi}
\newcommand{\e}{\varepsilon}

\newcommand{\oo}{\overline}

\newcommand{\ind}{\mathds{1}}
\newcommand{\n}[1]{\|#1\|}

\newcommand{\abs}[1]{\vert #1\vert}
\renewcommand{\leq}{\leqslant}
\renewcommand{\geq}{\geqslant}

\newcommand{\essinf}{\mathrm{ess\,inf}}
\newcommand{\esssup}{\mathrm{ess\,sup}}

\newcommand{\vertiii}[1]{{\left\vert\kern-0.25ex\left\vert\kern-0.25ex\left\vert #1 
    \right\vert\kern-0.25ex\right\vert\kern-0.25ex\right\vert}}

\newcounter{smallromans}

\newenvironment{romanenumerate}
{\begin{list}{{\normalfont\textrm{(\roman{smallromans})}}}%
  {\usecounter{smallromans}\setlength{\itemindent}{0cm}%
   \setlength{\leftmargin}{5.5ex}\setlength{\labelwidth}{5.5ex}%
   \setlength{\topsep}{.5ex}\setlength{\partopsep}{.5ex}%
   \setlength{\itemsep}{0.1ex}}}%
{\end{list}}

\newcounter{smallromansdash}

{\end{list}}

\newcounter{bigromans} 
  {\end{list}}

\begin{document}
\title[Steinhaus' lattice-point problem]{Steinhaus' lattice-point problem\\ for Banach spaces}

\author[T.~Kania]{Tomasz Kania}
\address{School of Mathematical Sciences, Western Gateway Building, University College Cork, Cork, Ireland \,\,{\rm and}\,\, Mathematics Institute, University of Warwick, Gibbet Hill Rd, Coventry, CV4 7AL, England}
\email{tomasz.marcin.kania@gmail.com, t.kania@warwick.ac.uk}

\author[T. Kochanek]{Tomasz Kochanek}
\address{Institute of Mathematics, Polish Academy of Sciences, \'Sniadeckich 8, 00-656 Warsaw, Poland\,\,{\rm and}\,\, Institute of Mathematics, University of Warsaw, Banacha~2, 02-097 Warsaw, Poland}
\email{tkoch@impan.pl}

\thanks{The first-named author acknowledges with thanks funding received from the European Research Council / ERC
Grant Agreement No. 291497. The research of the second-named author was supported by the Polish Ministry of Science and Higher Education in the years 2013--14, under Project No.~IP2012011072.}
\begin{abstract}
Steinhaus proved that given a~positive integer $n$, one may find a circle surrounding exactly $n$ points of the integer lattice. This statement has been recently extended to Hilbert spaces by Zwole\'{n}ski, who replaced the integer lattice by any infinite set that intersects every ball in at most finitely many points. We investigate Banach spaces satisfying this property, which we call (S), and characterise them by means of a~new geometric property of the unit sphere which allows us to show, {\it e.g.}, that all strictly convex norms have (S), nonetheless, there are plenty of non-strictly convex norms satisfying (S). We also study the corresponding renorming problem.
\end{abstract}

\maketitle
\section{Introduction and statement of the main results}
\noindent
The following feature of the integer lattice in the Euclidean plane was probably first observed by Steinhaus \cite[Problem 24 on p.~17]{steinhaus}: for any natural number $n$ one may find a~circle surrounding exactly $n$ lattice points. Zwole\'nski \cite{zwolenski} generalised this fact to the setting of Hilbert spaces in the following manner. He replaced the set of lattice points by a more general {\it quasi-finite} set, \emph{i.e.},~an~infinite subset $A$ of a metric space $X$ such that each ball in $X$ contains only finitely many elements of $A$. His result then reads as follows.
\begin{Theorem}[\cite{zwolenski}]\label{T1}
Let $A$ be a quasi-finite subset of a Hilbert space $X$. Then there exists a dense subset $Y\subset X$ such that for every $y\in Y$ and $n\in\N$ there exists a ball $B$ centred at $y$ with $\abs{A\cap B}=n$.
\end{Theorem}
Let us then distill the property that we will term {\it Steinhaus' property} (S). A metric space $X$ has this property if, by definition,
\begin{itemize*}
\item[(S)] for any quasi-finite set $A\subset X$ there exists a dense set $Y\subset X$ such that for all $y\in Y$ and $n\in\N$ there exists a ball $B$ centred at $y$ with $\abs{A\cap B}=n$.
\end{itemize*}
We translate condition (S), formulated above, into three equivalent statements concerning the geometry of the unit ball of a Banach space. Roughly speaking, they require that, locally, the unit sphere of $X$ does not look the same at any two distinct points. This approach will be particularly beneficial, as it will allow us to identify spaces that share that property with Hilbert spaces, yet of a very different nature. Our first main result then reads as follows. 

\begin{Theorem A}\label{thma}
Let $X$ be a Banach space. The following assertions are equivalent:
\begin{itemize*}
\item[(S)] $X$ has Steinhaus' property;
\item[(S$_1$)] for any quasi-finite set $A\subset X$ there exists a dense set $Y\subset X$ such that for every $y\in Y$ there exists a ball $B$ centred at $y$ with $\abs{A\cap B}=1$;
\item[(S')] for all $x,y\in X$ with $x\not=y$, $\n{x}=\n{y}=1$ and each $\delta>0$ there exists a $z\in X$ with $\n{z}<\delta$ such that one of the vectors $x+z$ and $y+z$ has norm greater than $1$, whereas the other has norm smaller than $1$;
\item[(S'')] for all $x,y\in X$ with $x\not=y$, $\n{x}=\n{y}=1$ and each $\delta>0$ there exists a $z\in X$ with $\n{z}<\delta$ such that $\n{x+z}\not=\n{y+z}$.
\end{itemize*}
\end{Theorem A}
In other words, condition (S'') means exactly that one cannot find a ‘neighbourhood’ of parallel line segments on the unit sphere of equal length. This seems to be a~new geometric property which, as we will see, is essentially weaker than strict convexity. Notice that, in contrast to many other classical properties, property (S) is not inherited by subspaces and, in a sense, is neither local nor global. \smallskip

Properties (S') and (S'') are related to another (weaker) property of `non-flatness' of the unit sphere:
\begin{itemize*}
\item[(F)] the unit sphere $S_X$ of $X$ does not contain any flat faces, that is to say, there is no non-empty subset of $S_X$, open in the relative norm topology, that is contained in a~hyperplane.
\end{itemize*}
Here by a~{\it hyperplane} of $X$ we understand a~translation of a~subspace of $X$ of codimension~$1$, {\it i.e.}, a~set of the form $x+\mathrm{ker}(x^\ast)$ for some $x\in X$ and $x^\ast\in X^\ast$. Note, however, that (F) does not imply (S'') that is witnessed by the norm $\n{(x,y,z)}=\max\{\sqrt{x^2+y^2},\abs{z}\}$ for $(x,y,z)\in\R^3$ (consider the points $(1,0,0)$ and $(1,0,\frac12)$). However, whether every Banach space admits a renorming satisfying (F) seems to be an~attractive open problem.
\smallskip

We employ the announced equivalence to extend Zwole\'nski's result to strictly convex Banach spaces (Corollary~\ref{C1}). It is well-known that not every Banach space admits a strictly convex renorming, just to mention the examples of $\ell_\infty(\Gamma)$ for any uncountable set $\Gamma$ (see \cite{scday} and \cite[\S 4.5]{diestel}) or the quotient space $\ell_\infty / c_0$ (\cite{bourgain}). This motivates the question of whether strict convexity and property (S) are equivalent at the level of renormings, and a~negative answer is a~part of our next result.

\begin{Theorem B}\label{T2} 
Assuming that the continuum is a real-valued measurable cardinal, there exists a~non-strictly convexifable Banach space whose norm satisfies (S). Moreover, for any Banach space $X$ we have:
\begin{romanenumerate}
\item if $\dim X\leqslant 2$, then $X$ has property (S) if and only if $X$ is strictly convex;

\vspace*{1mm}
\item if $\dim X>2$ and $X$ admits a~renorming with property (S), then it also admits a~non-strictly convex renorming with property (S).
\end{romanenumerate}
\end{Theorem B}
Solovay (\cite{solovay}) proved that the assertion that the continuum is a real-valued cardinal is equiconsistent with the existence of a two-valued measurable cardinal number, therefore its consistency cannot be proved in ZFC alone (assuming of course that ZFC itself is consistent). Interestingly, our construction in this universe is possible because the real-valued measurability of the continuum implies the failure of the Continuum Hypothesis (\cite[p.~131]{banachkuratowski}) and we take advantage not only of pleasant measure-theoretic properties of the continuum but also of the existence of an uncountable cardinal number below it.\smallskip

It seems unlikely that real-measurability of the continuum is really necessary to show that there exist Banach spaces with (S) but which do not have a strictly convex renorming. This leaves the question of possibility of such constructions in ZFC open.

\section{Proof of Theorem A}

\begin{proof}[Proof of Theorem A]
Since the implications (S) $\Rightarrow$ (S$_1$) and (S') $\Rightarrow$ (S'') hold true trivially, it is enough to prove that (S$_1$) $\Rightarrow$ (S'), (S') $\Rightarrow$ (S) and (S'') $\Rightarrow$ (S').

\medskip\noindent
{\bf (S$\boldsymbol{_1}$) $\boldsymbol{\Rightarrow}$ (S'): }Suppose that (S$_1$) holds. Fix any $\delta>0$ and $x,y\in X$ with $x\not=y$, $\n{x}=\n{y}=1$. Consider any quasi-finite set $A\subset X$ such that $A\cap (1+\delta)B_X=\{ x,y\}$, where $B_X$ stands for the closed unit ball of $X$. According to (${\rm S}_1$), there is a $u\in X$, $\n{u}<\delta/2$, such that for some $r>0$ the open ball $B(u,r)$ contains exactly one element of $A$. Suppose there is an $a\in A\setminus\{ x,y\}$ belonging to $B(u,r)$. Then $$r>\n{a-u}\geq\n{a}-\n{u}>(1+\delta)-\frac{\delta}{2}=1+\frac{\delta}{2}\, ,$$hence $\n{x-u}<r$, that is $x\in B(u,r)$; a contradiction. Consequently, $B(u,r)$ contains exactly one of the points $x$ and $y$, say $x\in B(u,r)$ and $y\not\in B(u,r)$. Then $$1-\frac{\delta}{2}<\n{x-u}<r\leq\n{y-u}<1+\frac{\delta}{2}\, .$$Suppose that $r\leq 1$, $r=1-\e$ with some $\e\in [0,\delta/2)$ and take any number $\rho$ satisfying $$0<\rho<\min\Bigl\{r-\n{x-u},\frac{\delta}{2}-\e\Bigr\}.$$ Obviously, we may find $v\in X$ with $\n{v}\leq\e+\rho$ such that $\n{y-(u+v)}\geq r+\e+\rho>1$. Then we also have $$\n{x-(u+v)}\leq\n{x-u}+\n{v}<r-\rho+\n{v}\leq 1.$$Therefore, setting $z=-(u+v)$ completes the proof of our claim, since we have the estimate $\n{u+v}<\e+\rho+\delta/2<\delta$. We proceed similarly in the case where $r>1$ so the proof of (S$_1$) $\Rightarrow$ (S') is then complete.\medskip

\noindent
{\bf (S') $\boldsymbol{\Rightarrow}$ (S): }Let $X$ be a Banach space $X$ that satisfies (S') and let $A\subset X$ be a quasi-finite set. For any $n\in\N$ set $$G_n=\bigl\{ x\in X\colon \abs{A\cap B(x,r)}=n\mbox{ for some }r>0\bigr\} .$$It is evident, in view of the definition of a quasi-finite set, that each $G_n$ is an open subset of $X$. We shall prove that it is also dense.\smallskip

Assume, in search of a contradiction, that there is an open ball $U=B(x_0, r_0)$ in $X$ not intersecting $G_n$. Rescaling $U$ if necessary, we may suppose that $A\cap U=\varnothing$. With any point $x\in U$ we associate two integers $m(x)<n$ and $k(x)\geq 2$ defined as follows: Since $x\not\in G_n$, there is the largest non-negative integer $m(x)<n$ for which there exists $q>0$ with $\abs{A\cap B(x,q)}=m(x)$. Then, for every $s>q$ we have either $\abs{A\cap B(x,s)}=m(x)$ or $\abs{A\cap B(x,s)}>n$. Define
$$
s=\inf\bigl\{ t>0\colon\abs{A\cap B(x,t)}>n\bigr\} .
$$
Then exactly $m(x)$ points $a_1,\ldots ,a_{m(x)}\in A$ lie in the ball $B(x,sr_0)$, whereas at least two such points lie on the boundary of $B(x,s)$; let us call them $b_1,\ldots ,b_k$, where $k\geq 2$. In this way we define $k(x)=k$.\smallskip

Now, we shall use an~infinite descent argument to obtain a~desired contradiction. Let $a_1,\ldots,a_m$, $b_1,\ldots,b_k$ be as above for $x=x_0$, where $m=m(x_0)$ and $k=k(x_0)$. Pick any $\delta>0$ such that 
\begin{equation*}
\begin{split}
\{ a_i\colon 1\leqslant i\leqslant m\}\subseteq B(x_0+u,s)\cap A \subseteq \{a_i,b_j &\colon 1\leqslant i\leqslant m, 1\leqslant j\leqslant k\}\\
&\mbox{for every }u\in X\mbox{ with }\n{u}<\delta.
\end{split}
\end{equation*}
Define $\rho=\max\{\n{a_i-x_0}\colon 1\leqslant i\leqslant m\}<s$ and set $\gamma=s-\rho$. Each of the vectors $(b_j-x_0)/s$ ($j=1,\ldots,k$) lies in the unit sphere. Applying the hypothesis (S') to any two of them (\emph{e.g.}, to $j=1,2$), we obtain a~point $z\in X$ with
$$
\n{sz}<\min\{\delta, \gamma/2\}
$$
such that one of the vectors: $b_j-x_0-sz$ ($j=1,2$) has norm greater than $s$, whereas the other has norm smaller than $s$. By decreasing $\delta$, if necessary, we may also assume that the point $x:=x_0+sz$ still lies in $U$. Therefore the ball $B(x,s)$ with the centre in $U$ contains all $a_i$'s ($1\leqslant i\leqslant m$) and at least one but not all among $b_j$'s ($1\leqslant j\leqslant k$). Observe also that by our choice of $z$, we have
$$
\n{a_i-x}\leqslant\n{a_i-x_0}+\n{sz}<\rho+\gamma/2=s-\gamma/2\quad\mbox{for each }1\leqslant i\leqslant m
$$
and
$$
\n{b_j-x}\geqslant\n{b_j-x_0}-\n{sz}>s-\gamma/2\quad\mbox{for each }1\leqslant j\leqslant k.
$$
Therefore, by suitably rescaling the ball $B(x,s)$, we obtain a~new ball centred at $x$ which contains all of $a_i$'s and whose boundary contains some but not all of $b_j$'s. This shows that we have either $m(x)>m(x_0)$ or $k(x)<k(x_0)$. This construction (with $x_0$ replaced by $x$) will ultimately lead to a~contradiction, as we finally arrive at a~point $u\in U$ with $m(u)\geqslant n$ or $k(u)<2$. Therefore, all the sets $G_n$ ($n\in\N$) are open and dense.\smallskip

By the Baire Category Theorem, the set $Y=\bigcap_{n=1}^\infty G_n$ is dense in $X$ and, obviously, for each $y\in Y$ and $n\in\N$ there is a ball $B$ centred at $y$ with $\abs{A\cap B}=n$. This completes the proof of (S).

\smallskip\noindent
{\bf (S'') $\boldsymbol{\Rightarrow}$ (S'):  }Assume the negation of (S') and choose distinct unit vectors $x,y\in X$ and $\delta>0$ so that there is no vector $z\in X$ with $\n{z}<\delta$ for which exactly one of the vectors $x+z$ and $y+z$ lies inside the unit ball of $X$. For every $u\in S_X$ define
$$
V_u=\bigl\{z\in S_X\colon \n{u+\alpha z}<1\mbox{ for some }\alpha>0\bigr\}
$$
and
$$
\lambda_u(z)=\min\bigl\{\delta,\,\inf\{\alpha>0\colon\n{u+\alpha z}\geq 1\}\bigr\}\quad (z\in V_u).
$$
By the assumption, we have $V_x=V_y$ and $\lambda_x(z)=\lambda_y(z)$ for every $z\in V_x$, which means that the unit sphere looks locally the same at $x$ and $y$ (via the translation by $y-x$), namely,
\begin{equation}\label{SS1}
y-x+(B(x,\delta)\cap S_X)=B(y,\delta)\cap S_X.
\end{equation}
Pick $\eta>0$ so small that
\begin{equation}\label{SS2}
\Biggl\|x-\frac{x+z}{\n{x+z}}\Biggr\|<\delta\quad\mbox{ and }\quad\Biggl\|y-\frac{y+z}{\n{y+z}}\Biggr\|<\delta\quad\mbox{ if }\,\n{z}<\eta.
\end{equation}

Now, using (S''), choose a~vector $z\in X$ with $\n{z}<\eta$ so that $\n{x+z}\not=\n{y+z}$. We have then two possibilities: either $\n{x+z}\leq 1$ and $\n{y+z}\leq 1$, or $\n{x+z}\geq 1$ and $\n{y+z}\geq 1$. We shall consider the former case; for the latter one the argument is similar.

With no loss of generality we can assume that $\n{x+z}>\n{y+z}$. Consider the function $g\colon [0,\infty)\to [0,\infty)$ given by $$
g(\alpha)=\n{x+z+\alpha(y-x)}$$
which is convex, as can be easily verified. In view of \eqref{SS1} and \eqref{SS2}, we have
$$
\Biggl\|y-x+\frac{x+z}{\n{x+z}}\Biggr\|=1,
$$
that is, $g(\n{x+z})=\n{x+z}$. We have also $g(0)=\n{x+z}$ and $g(1)=\n{y+z}<\n{x+z}$. This is a~contradiction with the convexity of $g$, as the arguments: $0$, $\n{x+z}$ and $1$ lie in this order on the real line.
\end{proof}
\section{Examples}
\noindent
In this section we will demonstrate some applications of Theorem~A in concrete situations. We begin with a~strengthening of Zwole\'nski's result. \smallskip

Given two elements $x,y$ in a real vector space $X$, we denote by $\overline{xy}$ the line segment between $x$ and $y$, \emph{i.e.}, $\overline{xy} = \{\lambda x + (1-\lambda)y\colon \lambda\in [0,1]\}$.
\begin{proposition}Let $X$ be a Banach space and suppose that $x,y\in X$ are distinct unit vectors. If $\overline{xy}\not\subseteq S_X$, then for each $\delta > 0$ there is $z\in X$ with $\|z\|<\delta$ such that one of the vector $x+z, y+z$ has norm greater than 1 whereas the other one has norm strictly less than 1.\end{proposition}
\begin{proof}
Let $\delta>0$ and $x,y\in X$ with $x\not=y$, $\n{x}=\n{y}=1$ be given. Then each point inside the segment $\oo{xy}$, joining $x$ and $y$, has norm smaller than $1$, whereas each point lying on the straight line passing through $x$ and $y$, but outside $\oo{xy}$, has norm larger than $1$. Therefore, any point $z\in X$ satisfying $0<\n{z}<\delta$ and $x+z\in\oo{xy}$ does the job. 
\end{proof}
\begin{corollary}\label{C1}
Every strictly convex Banach space $X$ satisfies {\rm (S)}.
\end{corollary}

Now, we will see that strictly convex spaces do not exhaust the whole class of Banach spaces satisfying Steinhaus' condition. In fact, these two classes differ already in dimension three. The following construction will also serve as a~base for the proof of Theorem~B.
\begin{example}\label{Ex}
We {\it claim} that there exists a~norm $\vertiii{\,\cdot\,}$ in $\R^3$ such that $(\R^3,\vertiii{\,\cdot\,})$ contains $\ell_\infty^2$ isometrically (and hence is not strictly convex), nonetheless it satisfies condition (S). We are indebted to the referee for suggesting the following example which significantly simplified our original construction.

First, observe that the negation of (S'') easily implies that there are two different points $x$ and $y$ on the~unit sphere and $\delta>0$ so that $\n{x+z}=\n{w+z}$ whenever $\n{z}<\delta$ and $w\in\overline{xy}$. In other words, if a~given Banach space fails Steinhaus' condition, then there must be a~`neighbourhood' of segments on the unit sphere. Having this in mind we set
$$
B=\{(x_1,x_2,x_3)\in [-1,1]^3\colon\abs{x_3}\leqslant f(x_1,x_2)\},
$$
where $f\colon [-1,1]^2\to [0,1]$ is any continuous function satisfying the equations $f(0,0)=1$ and $f(-x_1,-x_2)=f(x_1,x_2)$ which vanishes on the~boundary of $[-1,1]^2$ and is strictly concave on $(-1,1)^2$. For example, we can take $f(x_1,x_2)=(1-\abs{x_1})^p(1-\abs{x_2})^p$ with $0<p<\frac{1}{2}$. Then, let $\vertiii{\,\cdot\,}$ be the~norm on $\R^3$ defined as the~Minkowski functional of $B$. Since there are only four segments lying on the unit sphere (the edges of the~square $[-1,1]^2\times\{0\}$), the Banach space $(\R^3,\vertiii{\,\cdot\,})$ satisfies Steinhaus' condition due to the remark above.

It is worth noticing a~simple geometrical feature of $B$ which makes $\vertiii{\,\cdot\,}$ satisfy condition (S'). Namely, considering any two different points $\mathrm{x}=(t,1,0)$ and $\mathrm{y}=(u,1,0)$ with $0\leqslant t<u<1$ we see that the curve lying on $B$ that starts at $\mathrm{x}$ and is parallel to the~$x_2x_3$-plane is flatter at the~point $\mathrm{x}$ than its counterpart at the~point $\mathrm{y}$. Therefore, for a~given $\delta>0$, one can take a~vector $\mathrm{z}\in\R^3$ with $\n{\mathrm{z}}<\delta$ of the form $\mathrm{z}=(0,v,w)$ to guarantee that exactly one (more precisely: the~latter one) of the~vectors $\mathrm{x}+\mathrm{z}$, $\mathrm{y}+\mathrm{z}$ goes outside of $B$. For any other two points our claim is either trivial or analogous. The~upper part of the~ball $B$ defined as above with $p=\frac{1}{3}$, as well as some contour lines illustrating the~above-mentioned flattening effect, are depicted in the two figures below.

\begin{multicols}{2}
\includegraphics[scale=0.8]{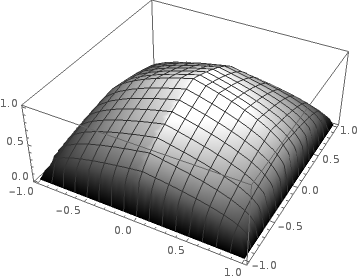}\vfill
\columnbreak
\hspace*{4mm}\includegraphics[scale=0.6]{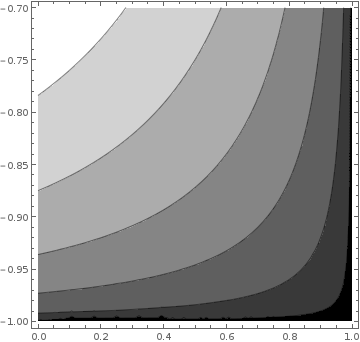}\vfill
\end{multicols}
\end{example}

\begin{remark}
The above example shows that there is a Banach space $X$ satisfying (S'), but containing a pair of distinct vectors $x,y\in X$, with $\n{x}=\n{y}=1$, such that for some $\delta>0$ it is impossible to increase $\n{x}$ and decrease $\n{y}$ by adding to $x$ and $y$ the same vector $z\in X$ with $\n{z}<\delta$. In other words, condition (S') cannot be strengthened by claiming which one of $x+z$ and $y+z$ has norm greater than $1$.
\end{remark}
\begin{remark}\label{remark_sc}
Of course, if $X$ is a non-strictly convex Banach space with $\dim X=2$, then condition (S') fails to hold. Therefore, Steinhaus' condition is equivalent to strict convexity in the class of Banach space with dimension at most $2$. 
\end{remark}

The next corollary demonstrates that the classical $L_1(\mu)$-spaces for atomless measures $\mu$ also satisfy Steinhaus' condition, giving thus another example of a non-strictly convex space with this property. Recall that a~set $A$ in a~measure space is called an~{\it atom} if $\mu(A)>0$ and $\mu(B)\in\{0,\mu(A)\}$ for every measurable subset $B$ of $A$.
\begin{proposition}\label{C2}
Let $(\Omega,\Sigma,\mu)$ be a~measure space. Then, the space $L_1(\mu)$ satisfies {\rm (S)} if and only if $\Omega$ contains at most one atom (up to measure-zero sets).
\end{proposition}
\begin{proof}
By Luther's theorem \cite{luther}, there is a decomposition $\mu=\mu_1+\mu_2$ with $\mu_1$ being semi-finite (\emph{i.e.},~for each $A\in\Sigma$ with $\mu_1(A)=\infty$, there is a subset $B\in\Sigma$ of $A$ such that $0<\mu_1(B)<\infty$) and $\mu_2$ being degenerate (\emph{i.e.},~the range of $\mu_2$ is contained in $\{0,\infty\}$). The space $L_1(\mu)$ is then isometrically isomorphic to $L_1(\mu_1)$ (for any $f\in L_1(\mu)$ we have $\mu_2(\{x\colon f(x)\not=0\})=0$, thus the identity map yields the desired isometry). Therefore, we consider only the case where $\mu$ is semi-finite.

\vspace*{2mm}
First, suppose that $(\Omega,\Sigma,\mu)$ is atomless. Fix two functions $f,g\in L_1(\mu)$ with $f\not=g$ and $\n{f}=\n{g}=1$, and let $\delta>0$ be given. Interchanging $f$ and $g$, if necessary, we may assume that there is a set $F\in\Sigma$ such that $0<\mu(F)<\infty$ and $f(\omega)>g(\omega)$ for $\omega\in F$. Since $$F=\bigcup_{n=1}^\infty\Bigl\{\omega\in F\colon f(\omega)>g(\omega)+\frac{1}{n}\Bigr\} ,$$we may also suppose that for some $\e>0$ and all $\omega\in F$ we have $f(\omega)>g(\omega)+\e$. Approximating $f$ and $g$ by step functions we may find a measurable set $F^\prime\subset F$ with $\mu(F^\prime)>0$ and some $c_f,c_g\in\R$ such that $$\abs{f(\omega)-c_f}<\frac{\e}{5}\quad\mbox{and}\quad\abs{g(\omega)-c_g}<\frac{\e}{5}\quad(\omega\in F^\prime).$$Hence, $c_f>c_g+\frac{3}{5}\e$ and $m_f>M_g+\frac{1}{5}\e$, where $m_f=\essinf\, f(F^\prime)$ and $M_g=\esssup\, g(F^\prime)$. We have three possibilities:
\begin{itemize*}
\item[(i)] $m_f>0$ and $M_g\geq 0$,
\item[(ii)] $m_f>0$ and $M_g<0$,
\item[(iii)] $m_f\leq 0$ and $M_g<0$.
\end{itemize*}
With no loss of generality suppose that either (i) or (ii) occurs (the case (iii) is analogous to (i)). Then there is a positive number $d$ such that $\abs{m_f-d}<m_f$ and $\abs{M_g-d}>\abs{M_g}$; indeed, in the former case we shall take any $d\in (2M_g,2m_f)$, while in the latter one any sufficiently small $d$ does the job. \smallskip

Now, observe that for almost all $\omega\in F^\prime$ we have
\begin{equation}\label{fg}
\abs{f(\omega)-d}<\abs{f(\omega)}\quad\mbox{and}\quad\abs{g(\omega)-d}>\abs{g(\omega)}.
\end{equation}
Indeed, for the first inequality note that in the case where $f(\omega)\geq d>0$ it holds trivially true, while in the opposite case we have $$\abs{f(\omega)-d}=d-f(\omega)\leq d-m_f\leq\abs{d-m_f}<m_f\leq\abs{f(\omega)}.$$For the other one observe that since $M_g<d$ (recall $\abs{M_g-d}>\abs{M_g}$), we have $g(\omega)<d$, thus in the case where $g(\omega)\geq 0$ we have $$\abs{g(\omega)-d}=d-g(\omega)\geq d-M_g=\abs{d-M_g}>\abs{M_g}\geq g(\omega)=\abs{g(\omega)},$$whereas in the case where $g(\omega)<0$ this inequality is trivial. \smallskip

By the Darboux property of finite atomless measures (\cite[Th\'eor\`eme, p.~16]{sierpinski}, see also \cite[\S 215]{fremlin}), there is a measurable set $H\subset F^\prime$ with $0<\mu(H)<\delta/d$. Then $\n{d\cdot\ind_H}<\delta$, where $\ind_H$ stands for the characteristic function of $H$, while inequalities \eqref{fg} imply that $\n{f-d\cdot \ind_H}<1$ and $\n{g-d\cdot \ind_H}>1$. This proves assertion (S'), and hence also (S).\smallskip

In the case where there is exactly one atom $A\subset\Omega$ (up to measure-zero sets), either $\mu(\Omega\setminus A)=0$, which means that $L_1(\mu)\cong\R$ isometrically, or there exists an~atomless part $B\subset\Omega$ of positive measure so that $\Omega=A\cup B$. In the latter case, fix any $f,g\in L_1(\mu)$ with $f\not=g$ and $\n{f}=\n{g}=1$. First, assume that $f\vert_B=g\vert_B$ outside a set of measure zero. Both $f$ and $g$ are constant almost everywhere on $A$; denote those constant values as $c_f$ and $c_g$, respectively. Since $\n{f}=\abs{c_f}\mu(A)+\int_B\abs{f}\,{\rm d}\mu$ and $\n{g}=\abs{c_g}\mu(A)+\int_B\abs{f}\,{\rm d}\mu$, we have $\abs{c_f}=\abs{c_g}$ and $c_f\not=c_g$. So, assuming that $c_f>0$ and $c_g<0$, for any given $\delta>0$ we have $\n{f+\delta\cdot \ind_A}>1$ and $\n{g+\delta\cdot\ind_A}<1$. In the case where $f\vert_B$ and $g\vert_B$ do not coincide almost everywhere, we repeat the argument from the first part of the proof for the atomless measure space $(B,\Sigma^\prime,\mu\vert_{\Sigma^\prime})$, where $\Sigma^\prime=\{B\cap C\colon C\in\Sigma\}$. Consequently, $L_1(\mu)$ has property (S) whenever the underlying measure space contains at most one atom.

\vspace*{2mm}
Finally, suppose $\Omega$ contains two disjoint atoms, say $A_1$ and $A_2$. Consider the functions $f=\frac{1}{2}(\ind_{A_1}+\ind_{A_2})$ and $g=\frac{1}{4}\ind_{A_1}+\frac{3}{4}\ind_{A_2}$. Obviously, for $\delta\in (0,\frac{1}{4})$ there is no function $h$ with $\n{h}<\delta$ so that exactly one of $f+h$ and $g+h$ has norm larger than~$1$, thus in this case (S) fails to hold.
\end{proof}

Theorem A gives an immediate answer to the question about Steinhaus' property for $C_0(K)$-spaces, and it is unsurprisingly negative except the trivial case where the considered space is one-dimensional.
\begin{corollary}\label{C3}
Let $K$ be a locally compact Hausdorff space that contains at least two points. Then the space $C_0(K)$ consisting of scalar-valued functions on $K$ vanishing at infinity does not have property {\rm (S)}.
\end{corollary}
\begin{proof}
Pick any two distinct points $u,v\in K$, and their disjoint neighbourhoods $U$ and $V$. Since $K$ is completely regular, there is a continuous map $\p\colon K\to [0,1]$ such that $\p(u)=1$ and $\p|_{K\setminus U}=0$. Similarly, since $K\setminus U$ is also completely regular, there is a continuous map $\p_1\colon K\setminus U\to [0,1/2]$ such that $\p_1(v)=1/2$ and $\p_1|_{K\setminus (U\cup V)}=0$. Then the mapping $\psi\colon K\to [0,1]$ defined by $$\psi(x)=\left\{\begin{array}{ll}\p(x) & \mbox{for }x\in U,\\ \p_1(x) & \mbox{for }x\in K\setminus U,\end{array}\right.$$is continuous and, of course, $\p\not=\psi$. So, both functions $\p$ and $\psi$ belong to the unit sphere of $C_0(K)$, but for any $\delta\in (0,1/2)$ condition (S') is violated. 
\end{proof}


\section{Proof of Theorem B}
Here, we shall construct a Banach space with property (S) but without any strictly convex renorming. Assume that the continuum $\mathfrak{c}$ is a real-valued cardinal number. This implies that there is an atomless, $\mathfrak{c}$-complete probability measure $\mu$ defined on the power set of a set $\Omega$ with the cardinality of the continuum (see, \emph{e.g.}, \cite[\S 543B(c)]{fremlin5})---here, by a $\lambda$-complete measure $\mu$ ($\lambda$ is an uncountable cardinal) we understand a measure satisfying the following condition: for every cardinal $\kappa<\lambda$ and for every family $(A_\alpha)_{\alpha<\kappa}$ of measurable sets, their union $A$ is measurable and $$\mu(A) = \sup\Big\{\mu\big(\bigcup_{\alpha\in F} A_\alpha\big)\colon F\subset \kappa \text{ finite}\Big\}.$$ The statement that the continuum is a real-valued cardinal is equiconsistent with the existence of a two-valued measurable cardinal (\cite{solovay}), which is stronger than the consistency of {ZFC} alone. Banach and Kuratowski (\cite[p.~131]{banachkuratowski}) proved that if such a measure exists, then the Continuum Hypothesis fails to hold, hence there exists at least one uncountable cardinal below the continuum. We will show that for any set $\Gamma$ with $|\Gamma|<\mathfrak{c}$, the Bochner space $L_1(\mu, \ell_\infty(\Gamma))$ satisfies Steinhaus' condition. In particular, if $\Gamma$ is uncountable, such space does not have a strictly convex renorming as it contains $\ell_\infty(\Gamma)$ embedded via constant functions and this space does not have such a renorming by a result of Day (\cite{scday}, see also \cite[\S 4.5]{diestel}).

\begin{theorem}\label{Bochner}Assume that $\mathfrak{c}$ is a~real-valued cardinal number and let $\Gamma$ be a~set with cardinality less than $\mathfrak{c}$. Then the Bochner space $X=L_1(\mu, \ell_\infty(\Gamma))$ has property {\rm (S)} for some atomless, probability measure $\mu$. \end{theorem}

\begin{proof}As $\mathfrak{c}$ is assumed to be a real-valued cardinal number, there exists an atomless probability measure space $(\Omega, \wp(\Omega), \mu)$, where $\Omega$ is a set with the cardinality of the continuum and $\mu$ is $\mathfrak{c}$-complete. Then $\mu$ is the required measure.\smallskip

Let $f\not=g$ be two norm-one elements of $X$. Since members of $X$ are equivalence classes of the relation of equality almost everywhere, let us work with concrete representatives $f,g\colon \Omega \to \ell_\infty(\Gamma)$. Fix $\delta>0$. There exists $n_0\in \mathbb{N}$ such that $\mu(F_{n_0})>0$ where
$$
F_{n_0} = \Big\{\omega\in \Omega \colon \|f(\omega) - g(\omega)\|_{\ell_\infty(\Gamma)} > \frac{1}{n_0}\Big\}.
$$
For each $\gamma\in \Gamma$ let 
$$
G_\gamma =\Big\{\omega\in F_{n_0}\colon |f(\omega)(\gamma) - g(\omega)(\gamma)| > \frac{1}{n_0}\Big\}.
$$
Since $\mu$ is defined on the power set of $\Omega$, there is no problem with measurability of the sets $G_\gamma$ ($\gamma\in \Gamma$). Also, as $\mu$ is $\mathfrak{c}$-complete and $|\Gamma|<\mathfrak{c}$, the set $G_{\gamma_0}$ has positive measure for some $\gamma_0\in \Gamma$. Interchanging $f$ with $g$, if necessary, we may suppose that the set
$$
F=\Bigl\{\omega\in G_{\gamma_0}\colon f(\omega)(\gamma_0) > g(\omega)(\gamma_0)+\frac{1}{n_0}\Bigr\}
$$
has positive measure. Now, we proceed as in the proof of Proposition~\ref{C2}.

Approximating the functions $\omega \mapsto f(\omega)(\gamma_0)$ and $\omega \mapsto g(\omega)(\gamma_0)$ ($\omega \in \Omega$) by step functions we may find a set $F^\prime\subset F$ with $\mu(F^\prime)>0$ and some $c_f,c_g\in\R$ such that for almost all $\omega\in F^\prime$ we have $$\abs{f(\omega)(\gamma_0)-c_f}<\frac{1}{5n_0}\quad\mbox{and}\quad\abs{g(\omega)(\gamma_0)-c_g}<\frac{1}{5n_0}.$$Hence, $c_f>c_g+\frac{3}{5n_0}$ and $m_f>M_g+\frac{1}{5n_0}$, where $m_f=\essinf\,f(F^\prime)$ and $M_g=\esssup\, g(F^\prime)$. We have then three possibilities:
\begin{itemize*}
\item[(i)] $m_f>0$ and $M_g\geq 0$,
\item[(ii)] $m_f>0$ and $M_g<0$,
\item[(iii)] $m_f\leq 0$ and $M_g<0$,
\end{itemize*}
which we tackle completely analogously as in the proof of Proposition~\ref{C2} (here $f(\omega)(\gamma_0)$ and $g(\omega)(\gamma_0)$ play the r\^{o}le of $f(\omega)$ and $g(\omega)$, respectively.) Therefore (assuming either (i) or (ii) holds true), we observe that for some $d>0$ and almost all $\omega\in F^\prime$ we have
\begin{equation}\label{fg2}
\abs{f(\omega)(\gamma_0)-d}<\abs{f(\omega)(\gamma_0)}\quad\mbox{and}\quad\abs{g(\omega)(\gamma_0)-d}>\abs{g(\omega)(\gamma_0)}.
\end{equation}

Since $\mu$ is atomless, there is a~measurable set $H\subset F^\prime$ with $0<\mu(H)<\delta/d$. Then $\n{d\!\cdot\!\delta_{\gamma_0}\!\cdot\! \ind_H}_X<\delta$, where $\delta_{\gamma_0}\in \ell_\infty(\Gamma)$ stands for the element that is zero apart from the $\gamma_0^{\rm th}$ coordinate where it assumes value~$1$. Hence, \eqref{fg2} imply that $\n{f-d\!\cdot\!\delta_{\gamma_0}\!\cdot\! \ind_H}_X<1$ and $\n{g-d\!\cdot\!\delta_{\gamma_0}\!\cdot\! \ind_H}_X>1$, as desired.\end{proof}

\begin{remark}Let us note that one may replace $L_1(\mu, \ell_\infty(\Gamma))$ in the statement of Theorem~\ref{Bochner} with $L_1(\mu, Y)$ where $Y$ is any subspace of $\ell_\infty(\Gamma)$ containing $c_0(\Gamma)$.\end{remark}

\begin{corollary}\label{cor_L1}
Under the assumptions of Theorem~\ref{Bochner}, for every uncountable set $\Gamma$ with cardinality less than the continuum, the Banach space $X=L_1(\mu, \ell_\infty(\Gamma))$ has {\rm (S)}, yet it lacks  a strictly convex renorming. \end{corollary}

We have thus proved the first assertion of Theorem~B; clause (i) has been also already observed; see Remark~\ref{remark_sc}. It remains to prove clause (ii).

\begin{proof}[Proof of Theorem B (continued)]
Assume that a~Banach space $X$ with $\dim X>2$ has a~norm $\n{\!\cdot\!}$ satisfying (S); we can assume that this norm is in fact strictly convex, as otherwise we are done. In the case where $\dim X=3$, the assertion is proved by Example~\ref{Ex}, so assume that $\dim X>3$. Choose any subspace $Y\subset X$ of codimension $2$ so that we have $X=Y\oplus\R\oplus\R$ and every element $x\in X$ may be typically written as $(y,\alpha,\beta)$ with $y\in Y$, $\alpha,\beta\in\R$. (In fact, the symbols $\R$ formally stand for some fixed one-dimensional subspaces of $X$.) Note that $Y$ is a~strictly convex space of dimension at least $2$. Let $\n{\!\cdot\!}^\prime$ be a~new norm on $X$ given by the decomposition $X=(Y\oplus_{\ell_1}\R)\oplus_{\ell_2}\R$, that is
$$
\n{x}^\prime=\sqrt{(\n{y}+\abs{\alpha})^2+\abs{\beta}^2}\qquad(x=(y,\alpha,\beta)).
$$
As any two finite direct sums of the same normed spaces are isomorphic, this defines an~equivalent norm on $X$ which obviously fails to be strictly convex. Next, we shall show that it has property (S).\smallskip

For, suppose $x_1=(y_1,\alpha_1,\beta_1)$ and $x_2=(y_2,\alpha_2,\beta_2)$ are two distinct points from the unit sphere of $(X,\n{\!\cdot\!}^\prime)$. If $\beta_1\not=\beta_2$, then $(\n{(y_1,\alpha_1)},\beta_1)$ and $(\n{(y_2,\alpha_2)},\beta_2)$ are two distinct points on the unit circle, where the norm symbol stands for the $\ell_1$-norm on $Y\oplus\R$. Thus, by manipulating the coordinates $\alpha$ and $\beta$ we obtain a~vector $z$ of the form $(0,\alpha,\beta)$, and of arbitrarily small length, so that $\n{x+z}^\prime\not=\n{y+z}^\prime$.\smallskip

Now, suppose that $\beta_1=\beta_2$ and hence $\n{y_1}+\abs{\alpha_1}=\n{y_2}+\abs{\alpha_2}$. If $y_1=y_2$, then it must be $\alpha_1=-\alpha_2\not=0$, whence we easily find a~desired vector $z$ being of the  form $(0,\alpha,0)$. So, assume we have $y_1\not=y_2$. In this case, we can find $z$ with the aid of following simple observation:\medskip

\noindent
{\it Claim. }Since $Y$ is strictly convex and $\dim Y\geq 2$, for every pair of distinct vectors $y_1,y_2\in Y$ and every $\delta>0$ there exists $z\in Y$ such that $\n{z}<\delta$ and $\n{y_1+z}-\n{y_1}\not=\n{y_2+z}-\n{y_2}$.\smallskip

\noindent
Indeed, if the vectors $y_1$ and $y_2$ are linearly independent, we take $z=\eta y_1$ for suitably small $\eta>0$. Then, $\n{y_1+z}-\n{y_1}=\eta\n{y_1}$ and this is equal to $\n{y_2+z}-\n{y_2}$ if and only if $\n{y_2+z}=\n{y_2}+\n{z}$, which is impossible as the norm is strictly convex. In the case where $y_2=\gamma y_1$ for some $\gamma\in\R$, we pick any vector $z$ that is linearly independent of $y_1$ and satisfies $\n{z}<\delta$. Then, assuming with no loss of generality that $\abs{\gamma}\geq 1$, the required condition becomes $\n{\gamma y_1+z}\not=(\abs{\gamma}-1)\n{y_1}+\n{y_1+z}$ which again follows from the strict convexity of $Y$. The claim has been thus proved.

\vspace*{2mm}
Now, take a vector $z\in Y$ as in the above claim. Then we have:
\begin{equation*}
\begin{split}
\n{y_1+(z,0,0)}^\prime &\not=\n{y_2+(z,0,0)}^\prime\Longleftrightarrow\\
\n{y_1+z}+\abs{\alpha_1} &\not=\n{y_2+z}+\abs{\alpha_2}\Longleftrightarrow\\
& \qquad\mbox{(because }\n{y_1}^\prime=\n{y_2}^\prime\mbox{ and }\beta_1=\beta_2\mbox{)} \\
\n{y_1+z}-\n{y_1} &\not=\n{y_2+z}-\n{y_2},
\end{split}
\end{equation*}
which is true. Therefore, we have checked that $(X,\n{\!\cdot\!}^\prime)$ satisfies condition (S'').
\end{proof}

\begin{remark}
The last proof shows that in all dimensions $\geq 4$ there exist easy examples of non-strictly convex Banach spaces satisfying (S). However, in the `$\ell_1$-$\ell_2$-{\it sum argument}' above we heavily used the assumption that $\dim X\geq 4$, so seemingly one cannot avoid a~bit more involved geometric argument in the case $\dim X=3$ (as in Example~\ref{Ex}).
\end{remark}

\end{document}